\documentclass{amsart}
\usepackage{graphicx}
\newtheorem{thm}{Theorem}[section]

\newtheorem{lem}[thm]{Lemma}

\theoremstyle{definition}
\newtheorem{defn}[thm]{Definition}
\theoremstyle{remark}
\newtheorem{rem}[thm]{Remark}
\numberwithin{equation}{section}

\begin{document}

\title{Rational Groups and a Characterization of a Class of Permutation Groups}%
\author{Cecil Andrew Ellard}%
\address{Bloomington, Indiana}%
\email{cellard@ivytech.edu}%

\thanks{}%


\begin{abstract}
We prove that a finite group is rational if and only if it has a set of permutation characters which separate conjugacy classes. It follows from this that a finite group is rational if and only if it has a representation as a permutation group in which any two elements fixing the same number of letters are conjugate.
\end{abstract}
\maketitle
\section{Rational Groups}
A group $G$ is said to be a $\textit{rational group}$ if for all $g$ $\in$ $G$ and for all complex irreducible characters $\chi$ of G, $\chi(g)$ is a  rational real number. The word ``irreducible" is not necesssary in the definition; a finite group is rational if and only if all complex characters (not just the irreducible ones) are rational valued. Since character values for any finite group are known to be complex algebraic integers, it follows that the character values of a rational group are actually rational integers, and so for a rational group, we have the unusual situation that the character table consists entirely of integers.\\

The finite Coxeter groups of type $A_{n}, B_{n}, D_{n}, E_{6}, E_{7}, E_{8}, F_{4},$ and $G_{2}$ (which include the finite symmetric groups $Sym(n)$), are rational groups. Rod Gow [4] has shown that only $Z_{2}$, $Z_{3}$, and $Z_{5}$ can arise as composition factors of a finite solvable rational group.  Walter Feit and Gary Seitz [3] have shown that the only non-abelian groups which occur as composition factors of a finite rational group are $Alt(n)$, $PSp(4,3)$, $Sp(6,2)$, $O^{+}(8,2)'$, $PSL(3,4)$, and $PSU(4,3)$. It follows from this that the only non-trivial finite simple rational groups are $Z_{2}$ of order 2, $Sp(6,2)$ of order $1,451,520=2^{9}\cdot 3^{4}\cdot 5 \cdot 7$, and $O^{+}(8,2)'$ of order $174,182,400=2^{12}\cdot 3^{5}\cdot 5^{2}\cdot 7$ [1]. One condition equivalent to rationality for finite groups is the following:

\begin{lem} A finite group $G$ is rational if and only if for each $g$ $\in$ $G$, all generators of the group $\langle g \rangle$ are conjugate in $G$.
\end{lem}

(See, for example, [5] or [2].)\\

If a finite group $G$ acts on a finite set $\Omega$, we can define a function $\theta :G \rightarrow \mathbb{C}$ by $\theta (g)=|Fix(g)|=|\{\alpha \in \Omega :g \alpha =\alpha\}|$.  $\theta (g)$ counts the number of fixed points of $g$ in $\Omega$. This function is a character of $G$ because if we  let $V$ be a complex vector space with basis $\Omega$, and extend the action of $G$ linearly to $V$ we obtain a $G$-module affording $\theta$. In particular, if $H$ is a subgroup of $G$, then $G$ acts transitively on the set of left cosets $xH$ of $H$ in $G$ by left multiplication. If we let  $1_{H}^{0}: G \rightarrow \mathbb{Z}$ denote the characteristic function of $H$, defined on $G$ by having value 1 on $H$ and 0 off $H$, then the character afforded by this action which counts the number of fixed cosets is also the induced character $1_{H}^{G}$ defined for $x \in G$ by 

\begin{displaymath}
  1_{H}^{G}(x) =\frac{1}{H} \sum_{g \in G} 1_{H}^{0}(g^{-1}xg).
\end{displaymath}\\

Notice that $1_{H}^{G}(x)=0$ if and only if $x$ is not conjugate to any element of $H$.

\begin{defn}Let $G$ be a group, $S$ a non-empty set, and let $\{ f_{i} : i \in I \}$ be a set of class functions, $f_{i}:G \rightarrow S$ (that is, for every $i$, if $g$ and $h$ are conjugate in $G$, then $f_{i}(g)=f_{i}(h)$, so that each $f_{i}$ is constant on conjugacy classes of $G$). We say that the set $\{ f_{i} : i \in I \}$ of functions $ \textit{separates conjugacy classes}$ of $G$ if whenever $g$ and $h$ are in distinct conjugacy classes of $G$, then there exists an $i$ such that $f_{i}(g) \neq f_{i}(h)$.
\end{defn}

\begin{rem} Note that a permutation character of $G$ separates conjugacy classes of $G$ if and only if the permutation representation affording that character has the property that any two elements of $G$ fixing the same number of letters are conjugate in $G$.
\end{rem}

\begin{rem} Let $G=Sym(3)$. $G$ acts by conjugation on its three involutions, and the permutation character of this action separates conjugacy classes, because the identity fixes three letters, the conjugacy class of 2-cycles fix one letter, and the conjugacy class of 3-cycles fix no letters.
\end{rem}

\begin{rem} Let $G=Sym(4)$. $G$ has exactly two classes $2_{1}$ and $2_{2}$ of involutions composed of the odd and even involutions, respectively. Let $\Omega$ be the union of these two classes. Thus, $|\Omega|=9$. $G$ acts by conjugation on $\Omega$, and the permutation character $\theta$ of this action separates conjugacy classes. We can explicitly describe $\theta$ as $\theta (g)=$ the number of involutions in $C_{G} (g)$. (In contrast, notice that the natural permutation character of $Sym(4)$ on $4$ letters does not separate conjugacy classes, since both a $4$-cycle and its square fix no letters, but are not conjugate in $Sym(n)$).
\end{rem}

The purpose of this note is to prove the following:\\

\begin{thm} Let $G$ be a finite group, and let $S=\{g_{1}, g_{2}, g_{3}, ..., g_{k}\}$ be a complete set of representatives of the conjugacy classes of $G$. Then the following are equivalent: \newline
\newline
(1) $G$ is a rational group.\newline
\newline
(2) The permutation characters $\{1_{<g_{i}>}^{G}: i=1,2,...,k\}$ separate conjugacy classes \newline 
\indent $\hskip 3pt$ of $G$. \newline
\newline
(3) $G$ has a finite collection $\mathcal{H}$ of subgroups  such that the permutation characters \newline 
\indent $\hskip 3pt$ $\{1_{H}^{G}: H \in \mathcal{H}\}$ separate conjugacy classes of $G$.\newline
\newline
(4) $G$ has a permutation character which separates the conjugacy classes of $G$. \newline
\indent $\hskip 3pt$ (That is, G has a permutation representation such that any two elements of $G$ \newline
\indent $\hskip 3pt$  fixing the same number of letters are conjugate in $G$.)

\end{thm}

We will prove the theorem using the following four lemmas:

\begin{lem} Let $G$ be a rational group, and let $S=\{g_{1}, g_{2}, g_{3}, ..., g_{k}\}$ be a complete set of representatives of the conjugacy classes of $G$. Then permutation characters $\{1_{<g_{i}>}^{G}: i=1,2,...,k\}$ separate conjugacy classes of $G$.
\end{lem}

\begin{proof} Assume that $G$ is a rational group. Let $g$ and $h$ be elements of $G$ which are not conjugate. We want to show that there exists a permutation character $1_{<g_{s}>}^{G}$ in the set $\{1_{<g_{i}>}^{G}: i=1,2,...,k\}$ such that $1_{<g_{s}>}^{G}(g) \neq 1_{<g_{s}>}^{G}(h)$. If $h$ is not conjugate to any element of $<g>$, then we have $1_{<g>}^{G}(g) \neq 0$, while $1_{<g>}^{G}(h) = 0$, so choosing $g_{s}$ conjugate to $g$, we get $1_{<g_{s}>}^{G}=1_{<g>}^{G}$, and we would be done. So assume that $h$ is conjugate to some element $h'$ of $<g>$. Since $g$ and $h'$ are not conjugate, Lemma 1.1 tells us that $h'$ is not a generator of $<g>$, and so the order of $g$ is greater than the order of $h'$, which is the same as the order of $h$.  Therefore, $g$ is not conjugate to any element of $<h>$. Therefore, $1_{<h>}^{G}(g) = 0$, while $1_{<h>}^{G}(h) \neq 0$, and so choosing $g_{s}$ conjugate to $h$, we get $1_{<g_{s}>}^{G}=1_{<h>}^{G}$, which proves the lemma.
\end{proof}

\begin{lem} Let $G$ be a finite group, and let $S=\{g_{1}, g_{2}, g_{3}, ..., g_{k}\}$ be a complete set of representatives of the conjugacy classes of $G$. Assume that the permutation characters $\{1_{<g_{i}>}^{G}: i=1,2,...,k\}$ separate conjugacy classes of $G$. Then $G$ has a finite collection $\mathcal{H}$ of subgroups such that the permutation characters $\{1_{H}^{G}: H \in \mathcal{H}\}$ separate conjugacy classes of $G$.
\end{lem}

\begin{proof} By assumption, we can let $\mathcal{H}=\{<g_{i}>: i=1,2,...,k\}$, since this is a finite collecton of subgroups of $G$, and the permutation characters $\{1_{<g_{i}>}^{G}: i=1,2,...,k\}$ separate conjugacy classes of $G$. 
\end{proof}

\begin{lem} Let $G$ be a finite group, and assume that $G$ has a finite collection $\mathcal{H}$ of subgroups such that the permutation characters $\{1_{H}^{G}: H \in \mathcal{H}\}$ separate conjugacy classes of $G$. Then $G$ has a permutation character which separates the conjugacy classes of $G$. That is, $G$ has a permutation representation such that any two elements of $G$ fixing the same number of letters are conjugate in $G$.
\end{lem}

\begin{proof} Let $G$ be a finite rational group and assume that $G$ has a finite collection $\mathcal{H}=\{H_{1}, H_{2}, H_{3}, ..., H_{n}\}$ of subgroups such that the permutation characters $\{1_{H}^{G}: H \in \mathcal{H}\}$ separate conjugacy classes of $G$.  Let $S=\{g_{1}, g_{2}, g_{3}, ..., g_{k}\}$ be a complete set of representatives of the conjugacy classes of $G$. Define $M$ by $M=max\{1, |1_{H}^{G}(g_{i})-1_{H}^{G}(g_{j})|: g_{i}, g_{j} \in S$ and $H \in \mathcal{H} \}$, so $M \in \mathbb{Z}$ and $M \ge  1$. Let $K$ be an integer greater than $M$. We claim the permutation character $\theta = K^{0}1_{H_{1}}^{G}+K^{1}1_{H_{2}}^{G}+...+K^{n-1}1_{H_{n}}^{G}$ separates conjugacy classes of $G$. To see this, let $g$ and $h$ be elements of $G$ and assume that $\theta(g)=\theta(h)$. We want to show that $g$ and $h$ are conjugate in $G$. By definition of $\theta$, 
\\

$(K^{0}1_{H_{1}}^{G}+K^{1}1_{H_{2}}^{G}+...+K^{n-1}1_{H_{n}}^{G})(g)=(K^{0}1_{H_{1}}^{G}+K^{1}1_{H_{2}}^{G}+...+K^{n-1}1_{H_{n}}^{G})(h)$.\\

By the definition of character addition, it follows that\\

$K^{0}1_{H_{1}}^{G}(g)+...+K^{n-1}1_{H_{n}}^{G}(g)=K^{0}1_{H_{1}}^{G}(h)+...+K^{n-1}1_{H_{n}}^{G}(h)$.\\

Since $g$ and $h$ are conjugate to elements $g_{i}$ and $g_{j}$ (respectively) in $S$, it follows that for every $H \in \mathcal{H}$ we have\\

 $\hskip 50pt$ $|1_{H}^{G}(g)-1_{H}^{G}(h)| = |1_{H}^{G}(g_{i})-1_{H}^{G}(g_{j})| \leq M < K$.\\

\noindent Since we have two base-$K$ expansions of a positive integer in which the coefficients differ by less than $K$, it follows that the corresponding coefficients of each are equal, and so $1_{H}^{G}(g)=1_{H}^{G}(h)$ for every $H \in \mathcal{H}$. Since the permutation characters $\{1_{H}^{G}: H \in \mathcal{H}\}$ separate conjugacy classes of $G$, and since they each agree when evaluated at $g$ and $h$, it follows that $g$ and $h$ are conjugate in $G$. This completes the proof of the lemma.
\end{proof}

\begin{lem} Assume that $G$ is a finite group, and that $G$ has a permutation character which separates the conjugacy classes of $G$ (that is, G has a permutation representation such that any two elements of $G$ fixing the same number of letters are conjugate in $G$.) Then $G$ is a rational group.
\end{lem}

\begin{proof} This proof is a variation on the theme of I.M. Isaacs' proof of the Burnside-Brauer Theorem [5]. Let $G$ be a finite group and let $\theta $ be a permutation character which separates the conjugacy classes of $G$. The idea behind the proof is simply that a system of linear equations over the integers which has a unique solution, has a solution consisting of rational numbers. Let $\chi$ be an arbitrary irreducible character of $G$, $\{g_{1}$, $g_{2}$, $g_{3}$, ..., $g_{k}\}$ a complete set of representatives of the conjugacy classes of $G$, and define complex numbers $c_{j}$ by $c_{j}=\sum_{h \varepsilon G}\theta^{j-1}(h)\chi (h^{-1})$, for $j=1,...,k$. Then using the definition of the inner product of characters of $G$, and the fact that $\chi (g_{i}^{-1}) = \overline{\chi (g_{i})}$,  $c_{j}=\sum_{i=1}^k \theta^{j-1}(g_{i})\chi (g_{i}^{-1})|Class(g_{i})|=|G|(\theta^{j-1} ,\chi )$. Consider the following system of $k$ linear equations in the $k$ unknowns $x_{i}$:\\

$ \theta^{j-1}(g_{1}) x_{1}+\theta^{j-1}(g_{2}) x_{2}+...+\theta^{j-1}(g_{k}) x_{k}=c_{j}$ for $j=1,...,k$\\

Since $\theta$ is a character of $G$, the inner products $(\theta^{j-1} ,\chi )$ are integers and therefore $c_{j}$ is an integer for all $j$. Also, since $\theta$ is integer valued, the coefficients of the system are integers. Therfore, we have a system of $k$ linear equations in $k$ unknowns over the integers. The determinant of the system is the Vandermonde determinant $\pm \prod_{i<j}(\theta (g_{i})-\theta (g_{j}))$, which is non-zero, since $\theta$ separates conjugacy classes of $G$. Therefore, the system has the unique rational solution $x_{i}=\chi (g_{i}^{-1})|Class(g_{i})|$ for $i=1,...,k$.  Thus $\chi (g_{j}^{-1})$ is rational, and therefore since $ \overline{\chi (g_{j})}= \chi (g_{j}^{-1})$, it follows that $\overline{\chi (g_{j})}$ is rational, and therefore, $\chi (g_{j})$ is rational too. So $\chi $ is rational valued on the set of representatives of the conjugacy classes. Finally, since $\chi$ is a class function, it is rational on all of $G$. Since $\chi$ was an arbitrary irreducible character of $G$, it follows that $G$ is a rational group.
\end{proof}

\begin{rem}
Let $\mathbb{Q}$ be the set of rational numbers. Recall that Artin's Theorem says that for any finite group, every character can be written as a $\mathbb{Q}$-linear combination of characters induced from representations of cyclic subgroups [5]. By the proof of Theorem 1.6 (or of Lemma 1.10), if $G$ is a finite rational group, and if \newline $\{g_{1}$, $g_{2}$, $g_{3}$, ..., $g_{k}\}$ a complete set of representatives of the conjugacy classes of $G$, then the $\mathbb{Z}$-linear span of $\{1_{\langle g_{i} \rangle}^{G}, i=1,...,k\}$ contains a permutation character, say $\theta $, that separates conjugacy classes of $G$. It follows that the characters $\theta^{0}, \theta^{1}, ...\theta^{k-1}$ are linearly independent over $\mathbb{Q}$, and therefore every rational valued class function on $G$ can be written as a $\mathbb{Q}$-linear combinaton of the characters $\theta^{0}, \theta^{1}, ...\theta^{k-1}$, and since these  powers of $\theta $ are in the  $\mathbb{Q}$-linear span of $\{1_{\langle g_{i} \rangle}^{G}, i=1,...,k\}$ it follows that every rational valued class function on $G$ can be written as a $\mathbb{Q}$-linear combinaton of $\{1_{\langle g_{i} \rangle}^{G}, i=1,...,k\}$. In particular, every character of a rational group $G$ can be written as a $\mathbb{Q}$-linear combination of characters induced from $\it{identity}$ representations of cyclic subgroups. In this sense, Theorem 1.6 may be viewed as a type of ``Artin's Theorem"  for finite rational groups.\\

\end{rem}

\end{document}